\documentclass{amsart}

\usepackage[T1]{fontenc}
\usepackage[utf8]{inputenc}
\usepackage{lmodern}

\theoremstyle{plain}
\newtheorem{theorem}{Theorem}[section]
\newtheorem{corollary}[theorem]{Corollary}
\newtheorem{claim}[theorem]{Claim}

\usepackage{mathtools}
\usepackage{enumitem}
\usepackage{graphicx}
\usepackage{float}
\usepackage{dsfont}

\usepackage{algorithm}
\usepackage{algpseudocode}

\usepackage{microtype}

\usepackage[colorlinks=true,linktocpage=true]{hyperref}
\usepackage[nameinlink]{cleveref}

\newcommand{\floor}[1]{\left\lfloor #1 \right\rfloor}
\newcommand{\ceil}[1]{\left\lceil #1 \right\rceil}
\newcommand{\ind}{\mathds{1}}

\title{Results on long twins in random words and permutations}
\author{Elliott Liu, Linus Tang, Jessica Wan}
\date{}

\begin{document}

\begin{abstract}
We study long $r$-twins in random words and permutations. Motivated by questions posed in works of Dudek–Grytczuk–Ruci\'nski, we obtain the following. For a uniform word in $[k]^n$ we prove sharp one--sided tail bounds showing that the maximum $r$-power length (the longest contiguous block that can be partitioned into $r$ identical subblocks) is concentrated around $\frac{\log n}{(r-1)\log k}$. For random permutations, we prove that for fixed $k$ and $r\to\infty$, a uniform permutation of $[rk]$ a.a.s.\ contains $r$ disjoint increasing subsequences of length $k$, generalizing a previous result that proves this for $k=2$. Finally, we use a computer-aided pattern count to improve the best known lower bound on the length of alternating twins in a random permutation to $\alpha_n \ge \left(\tfrac{1}{3}+0.0989-o(1)\right)n$, strengthening the previous constant.
\end{abstract}

\maketitle

\section{Introduction}
Questions pertaining to twin objects in words, permutations, and graphs have a rich history in the literature. 

For finite words on finite alphabets $[k] = \{1,2,\dots,k\}$, \emph{twins} are two (not necessarily contiguous) disjoint subsequences which are the same when both read as words. Axenovich, Person, and Puzynina \cite{axenovich2013twins} bounded $t(n,k)$, the minimum over all words $w$ of length $n$ over $[k]$ of the maximum length twins in $w$, showing in particular that $t(n,k) = n/k - o(n)$, which is tight when $k = 2$. They argue by partitioning the word into subwords with sufficiently uniform letter densities and provide a construction using the natural interlacing.

Later work \cite{bukh2016twins} strengthened the lower bound to $t(n,k) \geq 1.02 \cdot \frac{n}{k} - o(n)$ and $t(n,k) \geq \Omega(nk^{-2/3})$ and improved upper bounds for $k \geq 4$ to $t(n,4) \leq 0.4932n$, $t(n,5) \leq 0.48n$, and 
\begin{align*}
    t(n,k) \leq \left( \frac{e}{\sqrt{k}} - \frac{e^2 + 1/2}{k} + O(k^{-3/2}) \right)n + o(n).
\end{align*}

We focus our attention on twins in \emph{random} structures. 

{Dudek, Grytczuk, and Ruci\'nski}~\cite{dudek2023twins} first studied the maximum length of twins in a random word on the alphabet $[k]$. They prove that for a uniformly random word $W_k(n)\in[k]^n$, asymptotically almost surely

\begin{align*}
    t(W_k(n)) \geq \frac{1.64n}{k+1},
\end{align*}
for all $k\geq 3$. They ask about the length of the more restricted notions of twins, specifically \emph{shuffle squares} where the union of twins must form a contiguous segment. 

{Dudek, Grytczuk, and Ruci\'nski}~\cite{dudek2023twins} consider restricted versions of the problem of twins in random words, such as the length of a maximum \emph{$r$-power} in a random word; that is, the length of the largest set of $r$ ``twins” for which the union of those words forms a contiguous segment.

We show that the maximum length of an $r$-power in a random word of length $n$ over the alphabet $[k]$ is well-concentrated around \[\frac{\log n}{(r-1) \log k}.\]

We next examine twin structures in random permutations. Two sequences $(a_1,\dots,a_k)$ and $(b_1,\dots,b_k)$ of distinct integers are similar if their entries have the same relative order, i.e. for all $1\le i<j\le k$ we have $a_i<a_j$ if and only if $b_i<b_j$. Twins in a permutation are disjoint similar subsequences of the permutation.

A subproblem in twins in permutations involves up-and-down \emph{alternating} subsequences as examined by Dudek, Grytczuk, and Ruci\'nski in \cite{dudek2021perms}, specifically those where $\pi_{i_1}, \pi_{i_2}, \dots, \pi_{i_k}$ satisfy \[\pi_{i_1} > \pi_{i_2} < \pi_{i_3} > \cdots \quad \text{or} \quad \pi_{i_1} < \pi_{i_2} > \pi_{i_3} < \cdots.\] (Since we only consider the limiting behavior as the length $n$ of the permutation tends to infinity, it does not matter whether we only consider the first condition here.)

We define twins to simply be any pair of disjoint alternating subpermutations of the same length both satisfying this property.

A lower bound on the maximum length of such twins can be achieved by calculating the expected number of extrema in $\pi$, i.e. the number of indices $i$ for which $\pi_i > \max(\pi_{i-1}, \pi_{i+1})$ or $\pi_i < \min(\pi_{i-1}, \pi_{i+1})$, and summing those in $\pi'$ formed by removing all extrema of $\pi$. We give a simple application of Azuma-Hoeffding which provides asymptotic concentration. We extend this result by improving the constant factor from $\left(\frac{1}{3} + \frac{1}{60}\right)n$ to $\left( \frac{1}{3} + 0.0989 \right)n$ with a computer-aided calculation, which is much closer to the empirically optimal value when not changing the procedure used to construct the twins.

Dudek, Grytczuk, and Ruci\'nski in \cite{dudek2021perms} study questions related to tight multiple twins. Tight multiple $r$-twins in a permutation are $r$ similar subpermutations whose disjoint union forms a contiguous part of the original permutation, such as $2,1,3$ and $5,4,8$ in the permutation
\[(12,6,7,\boldsymbol{2,5,4,1,3,8},13,10,9,11).\]
They ask the question of whether a random permutation of length $rk$ almost surely contains $r$-twins of length $k$ as $k$ stays fixed and $r\to\infty$ (note that any such twins fill the entire permutation and are thus tight). We show that this is indeed the case; in other words, tight multiple twins of a fixed length can almost surely fill a large random permutation.

\section{Results}

Dudek, Grytczuk, and Ruci\'nski in \cite{dudek2023twins} ask about the expected maximum length of a square in a random word of length $n$ over a $k$--letter alphabet. We answer this question for general $r$-powers, supplying explicit tail bounds which determine the expected value up to the constant term.

\begin{theorem}\label{thm:m_tail}
Let $n,k,r$ be positive integers with $k,r\ge2$ and let $M_{n,k,r}$ be the maximum length of an $r$-power in a random word of length $n$ over alphabet $[k]$. Then there exists a universal constant $C>0$ such that for all real numbers $t\ge0$,
\[\Pr\left[M_{n,k,r}\ge\frac{\log n}{(r-1)\log k}+t\right]\le\exp(-(\log k)(r-1)t)\]
and
\[\Pr\left[M_{n,k,r}\le\frac{\log n}{(r-1)\log k}-t\right]\le\exp(-Ck^{(r-1)(t-1)}).\]
\end{theorem}
\begin{proof}
Let $m$ be a positive integer. Let $W$ be a random word of length $n$ over alphabet $[k]$. Consider the string $B$ of $n-m$ binary digits given by $B_i=\ind[W_i=W_{i+m}]$ for $i\in[n-m]$. Then each digit of $B$ equals $1$ with probability exactly $\frac1k$, independently from all other digits (note the pairs $(W_i,W_{i+m})$ are disjoint in $i$). Notice that for each $i\in[n-rm+1]$, an $r$-power of length $m$ begins at index $i$ of $W$ if and only if $B_i=\cdots=B_{i+(r-1)m-1}=1$; let $E_i$ denote this event, which occurs with probability $k^{-(r-1)m}$ for each $i$, though not independently. Here, the definitions of $B_i$ and $E_i$ depend on $m$, but we omit the $m$ in the notation for convenience.

We now prove the upper tail bound. Use $m=\ceil{\frac{\log n}{(r-1)\log k}+t}=\frac{\log n}{(r-1)\log k}+t'$ for some $t'\ge t$. Now

\begin{align*}
\Pr\left[M_{n,k,r}\ge\frac{\log n}{(r-1)\log k}+t\right]&=\Pr\left[M_{n,k,r}\ge m\right]\\
&\le\mathbb E\left[\sum_{i=1}^{n-rm+1}\ind[E_i]\right]\\
&=\sum_{i=1}^{n-rm+1}\Pr[E_i]\\
&\le nk^{-(r-1)m}\\
&=nk^{-(r-1)\left(\frac{\log n}{(r-1)\log k}+t'\right)}\\
&=k^{-(r-1)t'}\\
&\le k^{-(r-1)t}\\
&=\exp((-\log k)(r-1)t).
\end{align*}

We now prove the lower tail bound. This time we use $m=\floor{\frac{\log n}{(r-1)\log k}-t}+1=\frac{\log n}{(r-1)\log k}-t''$ for some $t''\ge t-1$. In the computation below, let $B_0=0$ for convenience.
\begin{align*}
\Pr&\left[M_{n,k,r}\le\frac{\log n}{(r-1)\log k}-t\right]\\
&=\Pr\left[M_{n,k,r}<m\right]\\
&=\Pr[(\forall i\le n-rm+1)\neg E_i]\\
&=\prod_{i=1}^{n-rm+1}\Pr[\neg E_i\mid(\forall j<i)\neg E_j]\\
&=\prod_{i=1}^{n-rm+1}(1-\Pr[B_{i-1}=0\mid(\forall j<i)\neg E_j]\Pr[E_i\mid (B_{i-1}=0)\land(\forall j<i)\neg E_j])\\
&=\prod_{i=1}^{n-rm+1}(1-\Pr[B_{i-1}=0\mid(\forall j<i)\neg E_j]\Pr[E_i])\\
&\le\prod_{i=1}^{n-rm+1}\left(1-\frac{k-1}k\cdot k^{-(r-1)m}\right)\\
&=\left(1-\frac{k-1}k\cdot k^{-(r-1)m}\right)^{(n-rm+1)}\\
&\le\exp\left(-\frac{k-1}k\cdot k^{-(r-1)m}(n-rm+1)\right)\\
&\le\exp\left(-\frac{k-1}k\cdot\frac{n-rm+1}n\cdot k^{(r-1)t''}\right)\\
&\le\exp\left(-\frac12\cdot\frac{n-\frac{r\log n}{(r-1)\log k}}n\cdot k^{(r-1)(t-1)}\right)\\
&\le\exp\left(-Ck^{(r-1)(t-1)}\right).
\end{align*}
\end{proof}
\begin{corollary}\label{cor:m_ev}
There exists a universal constant $C_1$ such that for all positive integers $n,k,r$ with $k,r\ge2$,
\[\left|\mathbb E[M_{n,k,r}]-\frac{\log n}{(r-1)\log k}\right|\le C_1.\]
\end{corollary}
\begin{proof}
Applying \Cref{thm:m_tail}, we have
\begin{align*}
\mathbb E[M_{n,k,r}]-\frac{\log n}{(r-1)\log k}&\le\int_0^\infty\Pr\left[M_{n,k,r}-\frac{\log n}{(r-1)\log k}\ge t\right]dt\\
&\le\int_0^\infty e^{-(\log k)(r-1)t}dt\\
&=\frac1{(\log k)(r-1)}
\end{align*}
and
\begin{align*}
-\left(\mathbb E[M_{n,k,r}]-\frac{\log n}{(r-1)\log k}\right)&\le\int_0^\infty\Pr\left[M_{n,k,r}-\frac{\log n}{(r-1)\log k}\le-t\right]dt\\
&\le\int_0^\infty e^{-Ck^{(r-1)(t-1)}}dt\\
&\le1-\frac{\log C}{(\log k)(r-1)}\int_0^\infty e^{-k^{(r-1)t}}dt\\
&\le1-\frac{\log C}{(\log k)(r-1)}\int_0^\infty e^{-t}dt\\
&=1-\frac{\log C}{(\log k)(r-1)},
\end{align*}
so when we set $C_1=\max(\frac1{\log 2},1-\frac{\log C}{\log 2})$, it holds that
\[\left|\mathbb E[M_{n,k,r}]-\frac{\log n}{(r-1)\log k}\right|\le C_1.\]
\end{proof}

Instead of asking about the maximum length $M$ of an $r$-power in a random word, we can also ask the converse: What is the maximum power $R$ of a given length $m$ that exists in a random word in $[k]^n$? The tail bounds proven in Theorem~\ref{thm:m_tail} also produce tail bounds on $R$.

\begin{corollary}\label{cor:r_tail}
Let $n,k,m$ be positive integers with $k\ge2$ and let $R_{n,k,m}$ be the maximum power of length $m$ in a random word in $[k]^n$. Then there exists a universal constant $C>0$ such that for all real numbers $u\ge0$,
\[\Pr\left[R_{n,k,m}\ge\frac{\log n}{m\log k}+1+u\right]\le\exp(-(\log k)mu)\]
and
\[\Pr\left[R_{n,k,m}\le\frac{\log n}{m\log k}+1-u\right]\le\exp(-Ck^{(m+1)u}n^{-1/m})\]
\end{corollary}
\begin{proof}
For any $n,k,m,r$, we have $\Pr[R_{n,k,\ceil m}\ge r]=\Pr[M_{n,k,\ceil r}\ge m]$. Plugging in $r=\frac{\log n}{m\log k}+1+u$ gives us
\begin{align*}
\Pr&\left[R_{n,k,m}\ge\frac{\log n}{m\log k}+1+u\right]\\
&=\Pr\left[M_{n,k,\ceil{\frac{\log n}{m\log k}+1+u}}\ge m\right]\\
&\le\exp\left(-(\log k)\left(\ceil{\frac{\log n}{m\log k}+u}\right)\left(m-\frac{\log n}{(\frac{\log n}{m\log k}+u)\log k}\right)\right)\\
&\le\exp\left(-(\log k)\left(\frac{\log n}{m\log k}+u\right)\left(\frac{m^2u\log k}{\log n+mu\log k}\right)\right)\\
&=\exp\left(-(\log k)mu\right).
\end{align*}

Similarly, for any $n,k,m,r$, we have $\Pr[R_{n,k,\floor m}\le r]=\Pr[M_{n,k,\floor r}\le m]$. We assume that $u$ is chosen so that $r=\frac{\log n}{m\log k}+1-u$ is an integer, because this is when the bound is the tightest. So we have
\begin{align*}
\Pr\left[R_{n,k,m}\le\frac{\log n}{m\log k}+1-u\right]&\le\Pr[M_{n,k,r}\le m]\\
&\le\exp(-Ck^{(r-1)(\frac{\log n}{(r-1)\log k}-m-1)})\\
&\le\exp(-Ck^{\frac{\log n}{\log k}-(m+1)(r-1)})\\
&\le\exp(-Ck^{\frac{\log n}{\log k}-(m+1)(\frac{\log n}{m\log k}-u)})\\
&\le\exp(-Ck^{(m+1)u}n^{-1/m}).
\end{align*}
\end{proof}

\begin{corollary}\label{cor:r_ev}
There exists a universal constant $C$ such that for all positive integers $n,k,m$ with $k\ge2$,
\[\mathbb E[R_{n,k,m}]-\left(\frac{\log n}{m\log k}+1\right)\le C.\]
\end{corollary}
\begin{proof}
We have
\begin{align*}
\mathbb E[R_{n,k,m}]-\left(\frac{\log n}{m\log k}+1\right)&\le\int_0^\infty\Pr\left[R_{n,k,m}-\left(\frac{\log n}{m\log k}+1\right)\ge u\right]du\\
&\le\int_0^\infty e^{-(\log k)mu}du\\
&=\frac1{m\log k}.
\end{align*}
\end{proof}

\textit{Remark.} We conjecture that $\mathbb E[R_{n,k,m}]-\left(\frac{\log n}{m\log k}+1\right)$ is also lower-bounded by a universal constant, but unfortunately the tail bound in \ref{cor:r_tail} is not strong enough to prove this due to the factor of $n^{-1/m}$ in the exponent.

\cite{dudek2021perms} defines $Q^{(r)}(k)$ as the number of permutations of $[rk]$ that are tight $r$-twins of length $k$. They prove that $\lim_{r\to\infty}\frac{Q^{(r)}(2)}{(2r)!}=1$. In other words, when $k=2$, a random permutation $[2r]$ is a.a.s.\ a tight $r$-twin of length $2$ as $r$ goes to $\infty$. We prove a generalization, replacing $2$ with an arbitrary positive integer $k$. Specifically, we show that, $\lim_{r\to\infty}\frac{Q^{(r)}(k)}{(kr)!}=1$ holds for any positive integer $k$, by considering increasing twins. This is restated in words and proven below.

\begin{theorem}\label{thm:perm_partition}
Let $k$ be a positive integer and $n=rk$ with $r\to\infty$. Then a random permutation $\Pi_n$ of $[n]$ a.a.s.\ contains $r$-twins of length $k$, each twin in increasing order.
\end{theorem}
\begin{proof}
Assume $k\ge2$ since the $k=1$ case is trivial.

Let $n=rk$ and consider a set $T$ of $n$ points chosen independently and uniformly at random from the interior of an $n\times n$ square $\mathcal S=\{(x,y)\in\mathbb R^2:0\le x,y\le n\}$. These points can be ordered in increasing order of $x$-coordinate, and they can be ordered in increasing order of $y$-coordinate. In this manner, the $n$ points determine a uniformly random permutation $\Pi_n$ of $[n]$, namely that which transforms the former order into the latter.

Then $\Pi_n$ contains $r$-twins of length $k$ with each twin in increasing order if and only if $T$ can be partitioned into $r$ sets of $k$ such that within each group the ``increasing $x$-coordinate” order and ``increasing $y$-coordinate” order are identical (we say that a set of $k$ points is \textit{ascending} if it has this property). It now suffices to show that $T$ a.a.s.\ can be partitioned into $r$ ascending sets.

A helpful intuition is that points near the bottom right and top left of $\mathcal S$ are the most restrictive in terms of the partitioning, so we deal with them first.

We divide $\mathcal S$ into a $t\times t$ grid of congruent square cells of side length $\frac nt$, where $t=\Theta(n^{4/13})$. We will further divide $\mathcal S$ into several regions and illustrate an example in the diagram below.

Consider the region consisting of $\frac{k(k+1)}2$ cells in a triangular formation in the bottom right of the grid. Divide it into two congruent parts, $\mathcal A$ and $\mathcal B$, along the line $x+y=n$. Similarly, construct their reflections $\mathcal C$ and $\mathcal D$ over the line $x=y$, which are located at the top right of the grid.

Now, construct rectangles $\mathcal A'$, $\mathcal B'$, $\mathcal C'$, $\mathcal D'$ along the right, bottom, top, and left of the square $\mathcal S$, respectively. The width of each rectangle is $n^{6/13}$.

Divide $\mathcal A'$ into a $k-1$ by $k-1$ grid of ``tall cells” using evenly spaced vertical lines and ``almost evenly spaced” horizontal lines. For the purposes of the argument that follows, it suffices to have the height of each ``tall cell” be $\Theta(n)$, and we assume for simplicity that each gridline of $\mathcal A'$ lies along a gridline of the initial $t\times t$ division of $\mathcal S$.

Divide $\mathcal B'$, $\mathcal C'$, and $\mathcal D'$ similarly.

\begin{figure}[h]
\begin{center}
\includegraphics[width=.8\linewidth]{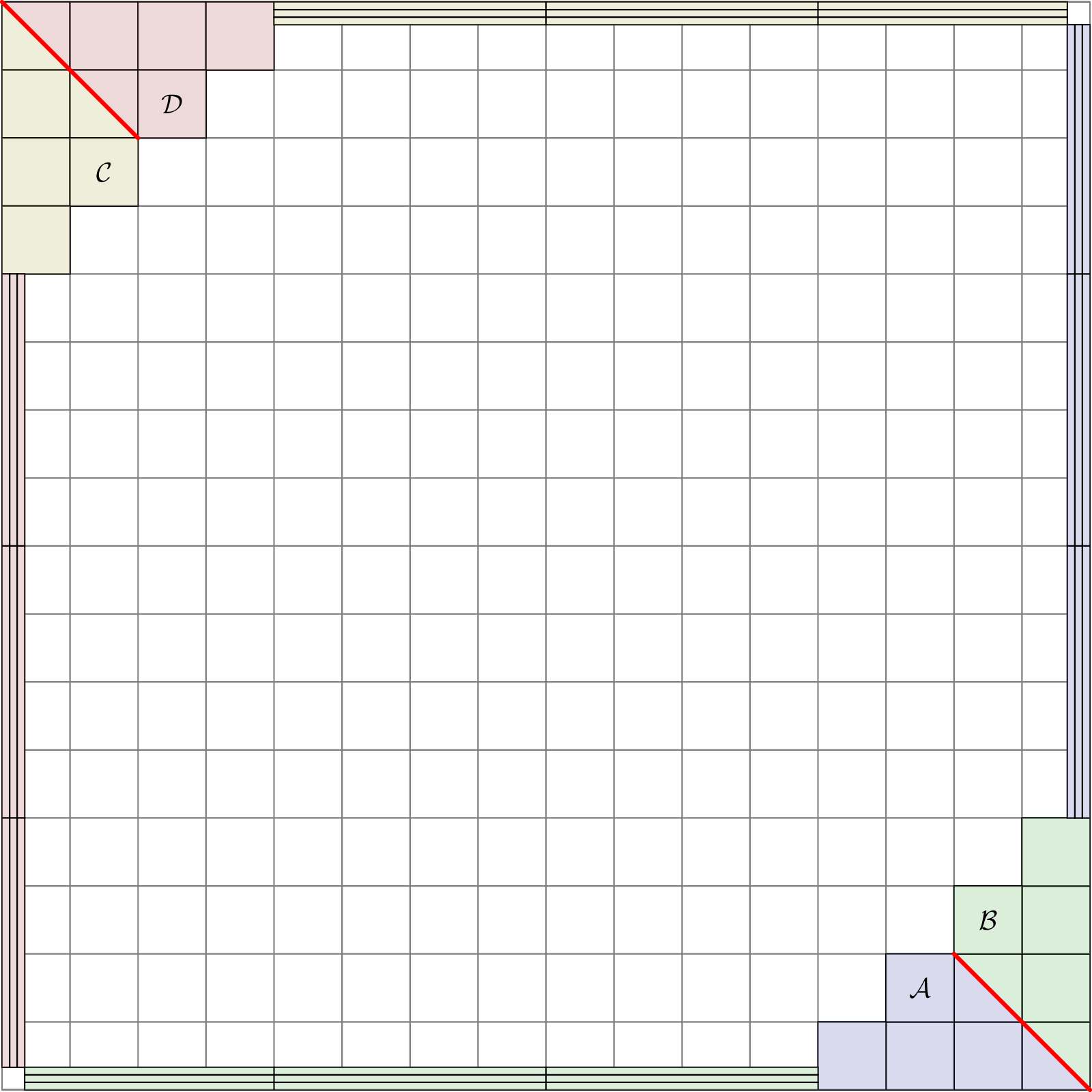}
\caption{An example with $t=16$, $k=4$}
\end{center}
\end{figure}

For a region $\mathcal R$ in the plane, we will use $|\mathcal R|$ as a shorthand for the number of points in $\mathcal R\cap T$.

To find a valid partition of the $n$ points, we will start by matching each point in $\mathcal A\cap T$ with $k-1$ points in $\mathcal A'\cap\ T$, and similarly for $\mathcal B$, $\mathcal C$, and $\mathcal D$. We will use the $k-1$ by $k-1$ grids to ensure that each set of $k$ points is ascending. In particular, for $i\in[k-1]$, let $\mathcal A'_i$ (respectively $\mathcal B'_i,\mathcal C'_i,\mathcal D'_i$) denote the tall cell which is $i$th from the bottom and $i$th from the left of the $k-1$ by $k-1$ grid $\mathcal A'$ (respectively, $\mathcal B',\mathcal C',\mathcal D'$). Let $\mathcal A'_{\mathrm{diag}}=\mathcal A'_1\cup\cdots\cup\mathcal A'_{k-1}$ and define $\mathcal B'_{\mathrm{diag}}$, $\mathcal C'_{\mathrm{diag}}$, and $\mathcal D'_{\mathrm{diag}}$ similarly. For each point in $\mathcal A\cap T$, we will want to match it with one point each from $\mathcal A'_1\cap T,\dots,\mathcal A'_{k-1}\cap T$ to form a set of $k$ (and similarly for $\mathcal B$, $\mathcal C$, and $\mathcal D$).

Naturally, this requires that $|\mathcal A|\le|\mathcal A'_i|$ for each $i\in[k-1]$. Below we show that this requirement a.a.s.\ holds, along with other requirements that will be used later in the proof. As a side note, although $\mathcal A$ has a much smaller area than each $\mathcal A'_i$ in the example diagram, it actually has a larger area as $n\to\infty$, which heuristically suggests that $|\mathcal A|\le|\mathcal A'_i|$.

For $i,j\in[t]$, let $\mathcal S_{i,j}$ denote the cell of grid $\mathcal S$ which is $i$th from the left and $j$th from the bottom.

\begin{claim}
All of the following hold a.a.s.:
\begin{enumerate}[label=(\arabic*)]
\item For all $i,j\in[t]$, we have $|\mathcal S_{i,j}|=\frac n{t^2}(1\pm o(1))=\Theta(n^{5/13})$.
\item $|\mathcal A|=\Theta(n^{5/13})$. Similarly for $\mathcal B$, $\mathcal C$, and $\mathcal D$.
\item For all $i\in[k-1]$, $|\mathcal A'_i|=\Theta(n^{6/13})$. Similarly for $\mathcal B$, $\mathcal C$, and $\mathcal D$.
\item For all $i,j\in[t]$ such that $\mathcal S_{i,j}$ overlaps with $\mathcal A'$, we have $|\mathcal S_{i,j}\cap\mathcal A'_{\mathrm{diag}}|=\Theta(n^{2/13})$. Similarly for $\mathcal B$, $\mathcal C$, and $\mathcal D$.
\item $|[0,n^{6/13}]\times[0,n^{6/13}]|=|[n-n^{6/13},n]\times[n-n^{6/13},n]|=0$. In other words, no point in $\mathcal S$ lies in the bottom left nor top right $n^{6/13}\times n^{6/13}$ corner of $\mathcal S$.
\end{enumerate}
\end{claim}
\begin{proof}
Each of the $n$ points is chosen uniformly randomly from square $\mathcal S$ of area $n^2$. So, for a given rectangle $\mathcal R\subseteq\mathcal S$, the probability that the point lies in $\mathcal R$ is $\frac{\mathrm{Area}(\mathcal R)}{n^2}$. In particular, $\mathbb E[|\mathcal R|]=\mathrm{Area}(\mathcal R)/n$. Furthermore, if $\mathbb E[|\mathcal R|]=\omega(1)$, then $\Pr\left[\left|\frac{|\mathcal R|}{\mathbb E[|\mathcal R|]}-1\right|\ge t\right]\le\exp\!\big(-t^2\,\mathbb E[|\mathcal R|]/3\big)$ for all $0\le t\le1$ by the Chernoff bound. Now it follows that (1), (2), (3), and (4) a.a.s.\ hold, because the relevant regions have areas
\begin{enumerate}[label=(\arabic*)]
\item $\frac nt\cdot\frac nt=\Theta(n^{18/13})$,
\item $\frac{k(k+1)}4\cdot\frac nt\cdot\frac nt=\Theta(n^{18/13})$,
\item $\Theta(n)\cdot\Theta(n^{6/13})=\Theta(n^{19/13})$, and
\item $\Theta(n^{9/13})\cdot\Theta(n^{6/13})=\Theta(n^{15/13})$.
\end{enumerate}
Finally, (5) a.a.s.\ holds by the first moment method because the relevant regions have area $n^{12/13}$ each.
\end{proof}

We now assume that (1) through (5) hold, and that $n$ is sufficiently large.

For each integer $d$ with $|d|\le t-k-1$, let $\mathcal S_d$ denote the union of cells along a diagonal of the grid $\mathcal S$:
\[\mathcal S_d=\bigcup_{\substack{i,j\in[t]\\i-j=d}} \mathcal S_{i,j}.\]

In particular, $\mathcal S=\mathcal A\sqcup\mathcal B\sqcup\mathcal C\sqcup\mathcal D\sqcup(\mathcal S_{-(t-k-1)}\sqcup\cdots\sqcup\mathcal S_{t-k-1})$.

As mentioned earlier, we match each point in $\mathcal A\cap T$ with one point each in $\mathcal A'_1\cap T,\dots,\mathcal A'_{k-1}\cap T$, which is possible because $|\mathcal A|=\Theta(n^{5/13})<\Theta(n^{6/13})=|\mathcal A_i'|$ by (2) and (3). Furthermore, each set is ascending, with the point chosen from $\mathcal A\cap T$ being the bottom-leftmost point, because no point in $\mathcal A\cap T$ has an $x$-coordinate greater than $n-n^{6/13}$ by (5).

We can similarly match each point in $\mathcal B\cap T$, $\mathcal C\cap T$, and $\mathcal D\cap T$ with $k-1$ points in $\mathcal B'$, $\mathcal C'$, and $\mathcal D'$, forming ascending sets of $k$ points.

Each $\mathcal A'_i$ intersects $\Theta(n^{4/13})$ cells $\mathcal S_{i,j}$. Since each such intersection has $\Theta(n^{2/13})$ points in $T$ by (4) and since $|\mathcal A|=\Theta(n^{5/13})$, it follows that we can select $\Theta(n^{1/13})$ points from each such intersection in the process of matching the points from $\mathcal A$. A similar argument applies to $\mathcal B$, $\mathcal C$, and $\mathcal D$. After matching some of the points in $\mathcal A'_{\mathrm{diag}}\cap T$, $\mathcal B'_{\mathrm{diag}}\cap T$, $\mathcal C'_{\mathrm{diag}}\cap T$, and $\mathcal D'_{\mathrm{diag}}\cap T$ with all points from $\mathcal A\cap T$, $\mathcal B\cap T$, $\mathcal C\cap T$, and $\mathcal D\cap T$, let $U\subseteq T$ be the set of points that remain unmatched. We can do so in such a way that the number of points in $\mathcal S_d\cap U$ is a multiple of $k$ for each $-(n-k-1)\le d\le n-k-1$, since each such diagonal intersects either $\mathcal A'$ or $\mathcal C'$.

Finally, it suffices to show that we can partition each $\mathcal S_d\cap U$ into ascending sets of $k$ points. For any $-(n-k-1)\le d\le n-k-1$, we have that $\mathcal S_d$ consists of $\ell>k$ cells, and that each cell intersects $T$ at $\frac n{t^2}(1\pm o(1))$ points by (1).

We now consider these $\ell$ cells in isolation. Since $O(n^{1/13})=o(\frac n{t^2})$ points from each cell have already been matched, we know that each cell also intersects $U$ at $\frac n{t^2}(1\pm o(1))$ points. We call these points unmatched points.

We can match the remaining points by repeatedly applying the following operation: Take $k$ cells with the most unmatched points, ties broken arbitrarily. Construct the ascending set containing one unmatched point from each cell. These points are no longer considered unmatched.

Notice that at the beginning of this process, no cell contains more than $\frac 1k$ of the unmatched points (because each of the $\ell>k$ cells has $\frac n{t^2}(1\pm o(1))$ unmatched points). We can verify that this property is preserved by the repeated operation, by casework on whether the cell in question was among those from which a point was matched. Therefore, the process terminates with all cells containing no more unmatched points, so we have demonstrated how to partition all of the remaining points from the diagonal sets $\mathcal S_d$ into ascending sets.

We can conclude that $T$ a.a.s.\ can be partitioned into ascending sets, so $\Pi_n$ a.a.s.\ contains $r$-twins of length $k$.
\end{proof}
\begin{corollary}
Let $k$ be a positive integer. Then the maximum multiplicity of twins of length $k$ in a random permutation $\Pi_n$ of $[n]$ is a.a.s.\ equal to $\floor{\frac nk}$.
\end{corollary}
\begin{proof}
By \Cref{thm:perm_partition}, the probability that the first $k\floor{\frac nk}$ positions of $\Pi_n$ are the union of $\floor{\frac nk}$-twins of length $k$ approaches $1$ as $n\to\infty$.
\end{proof}

Dudek, Grytczuk, and Ruci\'nski \cite{dudek2021perms} studied the maximum length $\alpha(n)$ of a pair of disjoint alternating subpermutations in $\Pi_n$. It is known that in a single permutation in $\Pi_n$, the longest alternating subpermutation has length the number of \emph{extremal} points, the local maxima and minima of the permutation, along with the endpoints, and Stanley \cite{stanley2008alternating} showed that this quantity is a.a.s.\ close to $2n/3$. 

For the lower bound, Dudek et al. produce a pair of disjoint alternating subpermutations as follows: split the $\pi$ in the middle, and assign the extremal points in the first half to a set $A$ and the second half to $B$. Then, consider the subpermutations of the two halves remaining after removing $A$ and $B$, respectively; the extremal points on the left may be added to $B$ and the extremal points on the right may be added to $A$.

Manually considering particular subpermutations, Dudek et al. prove that this procedure obtains 
\[\alpha_n \geq \left(1/3 + 0.0194\ldots - o(1)\right)n.\]
However, their computer simulations suggest that improvements up to $\alpha_n \geq (1/3 + 0.1006\ldots)n$ are possible without changing the procedure for producing the permutations.

Indeed, through some computer assistance we may show that
\begin{theorem}
    $\alpha_n \geq \left(1/3 + 0.0989\ldots + o(1)\right) n$
\end{theorem}

\begin{proof}
    Using a computer, we group the permutations $p$ of $\{1,2,\dots,k\}$ with $p_1 < p_2 > p_3 < \dots p_{k-1}$ and $p_{k-2}$, $p_{k-1}$, $p_k$ monotonic by the following characteristics simultaneously:
    \begin{itemize}
        \item $p_1$,
        \item the length $k$, and
        \item whether $p_{k-1} > p_1$.
    \end{itemize}
    Let $c_{p_1, k, X}$ be the number of such permutations.

    Call an index $r$ or a position $\pi_r$ \emph{sloped} if $\pi_{r-1}, \pi_r, \pi_{r+1}$ are monotonic. Observe that, as $n\to \infty$, for $\pi \sim \Pi_n$ with $\pi_r$ sloped, the probability that the closest sloped positions to the left and right are both higher or both lower than $\pi_r$ is exactly
    \[
    2 \cdot \sum_{k_1, k_2 \geq 3} \sum_{\substack{1 \leq p_1 \leq k_1 \\ 1 \leq p_2 \leq k_2}} \binom{k_1 - p_1 + p_2 - 1}{k_1 - p_1}\cdot \binom{k_2 - p_2 + p_1 - 1}{k_2 - p_2} \cdot \frac{2c_{p_1,k_1,\mathrm{True}} c_{p_2,k_2,\mathrm{False}}}{(k_1 + k_2 - 1)!},
    \]
    as the summation variables and the booleans for relative ordering provide sufficient statistics for the likelihood of the pattern's appearance. These fix the relative position of $\pi_r$ among the block of $k_1 + k_2 - 1$. Thus we find $\Pr[\pi_r$ is extremal after removing all peaks/lows$]$ by summing over possible pattern classes for the left and right subpermutations and enumerating the possible overall orderings.

    Via a computer computation, we find that this summation is at least $0.0989$ by taking only $k_1, k_2 \leq 13$. 

    We may see that $\alpha(n)$ is well-concentrated via the following standard Azuma-Hoeffding argument:

    Reveal the order of the $k$th prefix of $\Pi_n$ at the $k$th step, and the expectation of the size of the largest subpermutation forms a martingale. But note that for any permutation, shifting the position of the $k$th element in $A$ can only decrease the number of local extrema in $A$ by $3$, and this creates a direct correspondence between permutations of different $k$th-prefixes. Using that this correspondence is reversible, by symmetry changing the position of the $k$th element in the $k$th prefix can also only increase the lengths of $A$ or $B$ by $3$. 

    In particular, by Azuma-Hoeffding, the probability that the maximal lengths of disjoint alternating subpermutations $A$ and $B$ differ from the expectation by more than $d$ is at most $\exp\left(\frac{-d^2}{2 \cdot 3 \cdot n}\right)$, so as $n\to \infty$, with high probability $\alpha(n)$ differs from its mean by only $o(n)$, as desired.
\end{proof}
\section{Concluding Remarks}

We proved that the maximum power of length $m$ in a random word in $[k]^n$ $R_{n,k,m}$ satisfies the property that $\mathbb E[R_{n,k,m}]-\left(\frac{\log n}{m\log k}+1\right)$ is upper-bounded by a universal constant. We may also ask: is it also lower-bounded by a universal constant?

Next, we showed that the probability that a random permutation of $[rk]$ contains $r$-twins of length $k$ converges to $1$ as $k$ stays fixed and $r\to\infty$. Again, we may ask: what explicit lower bounds on this probability exist?

Finally, we showed that the length of the longest pair of disjoint alternating permutations in $\Pi_n$ $\alpha_n\ge(1/3+0.0989\dots)n$ using a computer-assisted approach. (Pseudocode may be found in the appendix.) However, it is still unclear what the greatest constant $c$ such that $\alpha_n\ge(c-o(1))n$ is. Dudek et al. \cite{dudek2021perms} upper bounds $\alpha_n\leq (1/2 - 1/120 + o(1)) n$, slightly improving upon the trivial bound $n/2$. It is likely possible to further improve the lower bound by using a different method to construct the alternating twins.

\noindent\textbf{Acknowledgments.} This research was done at the Massachusetts Institute of Technology. We thank Kuikui Liu for this opportunity.

\bibliographystyle{alpha}
\bibliography{sample}

\section{Appendix}
\subsection{Pseudocode for the alternating-twins computation}

We collect the $c_{p_1, k, X}$ as follows:

\begin{algorithm}[H]
  \caption{\textproc{CountGood}$(n)$}
  \begin{algorithmic}[1]
    \Require $n \ge 3$
    \State $C \gets \text{empty map from keys to integers}$
    \For{$k \gets 3$ \textbf{to} $n$}                         \Comment{permutation length}
      \State $Numbers \gets \{1,\dots,k\}$
      \ForAll{$p \in$ \Call{Permutations}{$Numbers$}}
        \If{$p_2 > p_1$}                                      \Comment{keep only ascending first pair}
          \If{\Call{CheckGood}{$p$}}
            \State $key \gets (p_1,\,k,\, [\,p_{k-1} > p_1\,])$
            \State $C[key] \gets C[key] + 1$
          \EndIf
        \EndIf
      \EndFor
    \EndFor
    \State \Return $C$
  \end{algorithmic}
\end{algorithm}

\vspace{1em}

\begin{algorithm}[H]
  \caption{\textproc{CheckGood}$(p)$}
  \begin{algorithmic}[1]
    \Require $|p| \ge 3$
    \State $k \gets |p|$
    \For{$i \gets 2$ \textbf{to} $k-2$}        
      \If{$i$ \textbf{even} \textbf{and} $\neg(p_i>p_{i-1}\,\land\,p_i>p_{i+1})$}
        \State \Return \textbf{false}
      \ElsIf{$i$ \textbf{odd} \textbf{and} $\neg(p_i<p_{i-1}\,\land\,p_i<p_{i+1})$}
        \State \Return \textbf{false}
      \EndIf
    \EndFor
    \State \Return $(p_{k-3}<p_{k-2}<p_{k-1}) \;\lor\; (p_{k-3}>p_{k-2}>p_{k-1})$
  \end{algorithmic}
\end{algorithm}

\end{document}